\documentclass{amsart}

\usepackage{amssymb,stmaryrd,mathrsfs,dsfont,amsmath}
\usepackage{colonequals} %% For \colonequals as needed in the body. Compare := and := carefully!!!

\theoremstyle{plain}
\newtheorem{theorem}{Theorem}[section]
\newtheorem{lemma}[theorem]{Lemma}
\newtheorem{proposition}[theorem]{Proposition}

\theoremstyle{definition}

\theoremstyle{remark}

\DeclareMathOperator{\reg}{\mbox{reg}}
\DeclareMathOperator{\ureg}{\mbox{ureg}}
\DeclareMathOperator{\codim}{\mbox{codim}}
\DeclareMathOperator{\nul}{\mbox{nullity}}
\DeclareMathOperator{\corank}{\mbox{corank}}
\DeclareMathOperator{\sym}{\mbox{Sym}}
\DeclareMathOperator{\aut}{\mbox{Aut}}

\input xy
\xyoption{all}

\begin{document}

\title[Semigroups of (linear) transformations whose restrictions belong \ldots]
{Semigroups of (linear) transformations whose restrictions belong to a given semigroup}

\author[Mosarof Sarkar]{\bfseries Mosarof Sarkar}
\address{Department of Mathematics, Central University of South Bihar, Gaya, Bihar, India}
\email{mosarofsarkar@cusb.ac.in}

\author[Shubh N. Singh]{\bfseries Shubh N. Singh}
\address{Department of Mathematics, Central University of South Bihar, Gaya, Bihar, India}
\email{shubh@cub.ac.in}

%\date{...}

\begin{abstract}
Let $T(X)$ (resp. L(V)) be the semigroup of all transformations (resp. linear transformations) of a set $X$ (resp. vector space $V$). For a subset $Y$ of $X$ and a subsemigroup $\mathbb{S}(Y)$ of $T(Y)$, consider the subsemigroup $T_{\mathbb{S}(Y)}(X) = \{f\in T(X)\colon f_{\upharpoonright_Y} \in \mathbb{S}(Y)\}$ of $T(X)$, 
where $f_{\upharpoonright_Y}\in T(Y)$ agrees with $f$ on $Y$. We give a new characterization for $T_{\mathbb{S}(Y)}(X)$ to be a regular semigroup [inverse semigroup]. 
For a subspace $W$ of $V$ and a subsemigroup $\mathbb{S}(W)$ of $L(W)$,
we define an analogous subsemigroup $L_{\mathbb{S}(W)}(V) = \{f\in L(V) \colon f_{\upharpoonright_W} \in \mathbb{S}(W)\}$ of $L(V)$. We describe regular elements in $L_{\mathbb{S}(W)}(V)$ and determine when $L_{\mathbb{S}(W)}(V)$ is a regular semigroup [inverse semigroup, completely regular semigroup]. If $\mathbb{S}(Y)$ (resp. $\mathbb{S}(W)$) contains the identity of $T(Y)$ (resp. $L(W)$), we describe unit-regular elements in $T_{\mathbb{S}(Y)}(X)$ (resp. $L_{\mathbb{S}(W)}(V)$) and determine when $T_{\mathbb{S}(Y)}(X)$ (resp. $L_{\mathbb{S}(W)}(V)$) is a unit-regular semigroup.  
\end{abstract}

\subjclass[2010]{20M17, 20M20, 15A03, 15A04.}

\keywords{Semigroups of transformations; Regular and unit-regular elements; Regular semigroups;  Unit-regular semigroups;  Completely regular semigroups;  Inverse semigroups.}

\maketitle

\section{Introduction}

For a given set $X$, denote by $T(X)$ the semigroup under composition consisting of all transformations from $X$ into itself.  The semigroup $T(X)$ and its various subsemigroups have been extensively studied over the last several decades, since every semigroup can be embedded in $T(Z)$ for some set $Z$ (cf. \cite[Theorem 1.1.2]{howie95}).

%%%%%%%%%%%%%%%%%%%%%%%%%%%%
%%%%%%%%%%%%%%%%%%%%%%%%%%%%
%%%%%%%%%%%%%%%%%%%%%%%%%%%%
%%%%%%%%%%%%%%%%%%%%%%%%%%%%
%OK---10-03-2022

\vspace{0.05cm}
Recently, Janusz Konieczny \cite{koni-sf22} introduced the subsemigroup $T_{\mathbb{S}(Y)}(X)$ of $T(X)$ defined by
\[T_{\mathbb{S}(Y)}(X) = \{f\in T(X)\colon f_{\upharpoonright_Y} \in \mathbb{S}(Y)\},\] 
where $Y$ is a subset of $X$, $\mathbb{S}(Y)$ is a subsemigroup of $T(Y)$, and $f_{\upharpoonright_Y}\colon Y \to Y$ is the transformation that agrees with $f$ on $Y$. Here, Janusz Konieczny described regular elements in $T_{\mathbb{S}(Y)}(X)$ and determined when $T_{\mathbb{S}(Y)}(X)$ is a regular semigroup and when it is an inverse semigroup. The authors \cite{koni-sf22, shubh-sf22-2} also determined when $T_{\mathbb{S}(Y)}(X)$ is a completely regular semigroup. Janusz Konieczny \cite{koni-sf22} characterized Green's relations on $T_{\mathbb{S}(Y)}(X)$ if $\mathbb{S}(Y)$ contains the identity transformation $I_Y$ on $Y$. If $Y = \varnothing$, then it is clear that $T_{\mathbb{S}(Y)}(X) = T(X)$. For $Y\neq \varnothing$, the semigroup $T_{\mathbb{S}(Y)}(X)$ has also been studied for several specific subsemigroups $\mathbb{S}(Y)$ of $T(Y)$ (see, e.g., \cite{hony11, magill66, nenth05, shubh-aejm22, sun-ams15, sun-as13}).

%%%%%%%%%%%%%%%%%%%%%%%%%%%%
%%%%%%%%%%%%%%%%%%%%%%%%%%%%
%%%%%%%%%%%%%%%%%%%%%%%%%%%%
%%%%%%%%%%%%%%%%%%%%%%%%%%%%
%OK---10-03-2022

%which was introduced in \cite{nenth-ijmms06}
\vspace{0.05cm}
For a given vector space $V$ over a field, denote by $L(V)$ the semigroup under composition consisting of all linear transformations from $V$ into itself.
Inspired by the construction of the subsemigroup $T_{\mathbb{S}(Y)}(X)$ of $T(X)$, we define here an analogous subsemigroup $L_{\mathbb{S}(W)}(V)$ of $L(V)$ by
\[L_{\mathbb{S}(W)}(V) = \{f\in L(V) \colon f_{\upharpoonright_W} \in \mathbb{S}(W)\},\] where $W$ is a subspace of $V$, $\mathbb{S}(W)$ is a subsemigroup of $L(W)$, and $f_{\upharpoonright_W}\colon W \to W$ is the linear transformation that agrees with $f$ on $W$. The semigroup $L_{\mathbb{S}(W)}(V)$ generalizes both $L(V)$ and $\mathbb{S}(W)$ in the sense that $L_{\mathbb{S}(W)}(V) = L(V)$ if $W$ is the zero subspace of $V$, and $L_{\mathbb{S}(W)}(V) = \mathbb{S}(W)$ if $W = V$. The semigroup $L_{\mathbb{S}(W)}(V)$ is also a generalization of the subsemigroup $\overline{L}(V, W) = \{f\in L(V)\colon Wf\subseteq W\}$ of $L(V)$, since $L_{L(W)}(V)= \overline{L}(V, W)$. The semigroup 
$\overline{L}(V, W)$ has been studied by several authors (see, e.g., \cite{chaiya-sf19, nenth-ijmms06, shubh-lma22}). 

%\cite{chinar-im18}
%%%%%%%%%%%%%%%%%%%%%%%%%%%%
%%%%%%%%%%%%%%%%%%%%%%%%%%%%
%%%%%%%%%%%%%%%%%%%%%%%%%%%%
%%%%%%%%%%%%%%%%%%%%%%%%%%%%

\vspace{0.05cm}
The purpose of this paper is to study both the semigroups  $T_{\mathbb{S}(Y)}(X)$ and $L_{\mathbb{S}(W)}(V)$. This paper is organized as follows. Section $2$ consists of the definitions and notation used throughout this paper. In Section $3$, we give a new characterization for the regularity of $T_{\mathbb{S}(Y)}(X)$. Next, we give a new characterization for $T_{\mathbb{S}(Y)}(X)$ to be an inverse semigroup. If $\mathbb{S}(Y)$ contains the identity of $T(Y)$, then we characterize unit-regular elements in $T_{\mathbb{S}(Y)}(X)$ and determine when $T_{\mathbb{S}(Y)}(X)$ is a unit-regular semigroup. In Section $4$, we characterize regular elements in $L_{\mathbb{S}(W)}(V)$ and determine when $L_{\mathbb{S}(W)}(V)$ is a regular semigroup. Next, if $\mathbb{S}(W)$ contains the identity of $L(W)$, then we characterize unit-regular elements in $L_{\mathbb{S}(W)}(V)$ and determine when $L_{\mathbb{S}(W)}(V)$ is a unit-regular semigroup. We also determine when $L_{\mathbb{S}(W)}(V)$ is an inverse semigroup, and when $L_{\mathbb{S}(W)}(V)$ is a completely regular semigroup.    

%%%%%%%%%%%%%%%%%%%%%%%%%%%%
%%%%%%%%%%%%%%%%%%%%%%%%%%%%
%%%%%%%%%%%%%%%%%%%%%%%%%%%%
%%%%%%%%%%%%%%%%%%%%%%%%%%%%
%\newpage

\section{Preliminaries and Notation}

%\vspace{0.1cm}
%Throughout this paper, we will use the definitions and notation from \cite{howie95} and \cite{s-roman07} unless otherwise specified. 
%Let $S$ be a semigroup. We denote by $E(S)$ the set of all idempotents of $S$. 
%
%
%By a \emph{semigroup}, we mean a nonempty set equipped with a binary operation. Recall that a completely regular semigroup is a semigroup in which every element is in some subgroup of the semigroup.
%

%%%%%%%%%%%%%%%%%%%%%%%%%%%%%
%%%%%%%%%%%%%%%%%%%%%%%%%%%%%
%%%%%%%%%%%%%%%%%%%%%%%%%%%%%
%%%%%%%%%%%%%%%%%%%%%%%%%%%%%

We denote the function composition by juxtaposition, and compose functions from left to right. Let $f\colon A \to B$ be a map. We write the image of $x\in A$ under $f$ by $x f$, and the image set of $f$ by $R(f)$. For any subset $C$ of $B$, let $Cf^{-1} = \{x\in A\colon xf \in C\}$. If $C = \{z\}$, then we write $Cf^{-1}$ more simply as $zf^{-1}$. The \emph{kernel} of $f$ is an equivalence relation $\ker(f)$ on $A$ defined by $\ker(f) = \{(x,y)\in A\times A \colon xf = yf\}$.

%For any $f,g\colon X \to X$, denote by $fg$ the transformation of $X$ obtained by performing first $f$ and then $g$.

%%%%%%%%%%%%%%%%%%%%%%%%%%%%%
%%%%%%%%%%%%%%%%%%%%%%%%%%%%%
%%%%%%%%%%%%%%%%%%%%%%%%%%%%%
%%%%%%%%%%%%%%%%%%%%%%%%%%%%%
\vspace{0.05cm}

Let $A$ be any set. A \emph{transversal} of an equivalence relation $\rho$ on $A$ is a subset of $A$ which contains exactly one element of each $\rho$-class. We use $T_f$ for any transversal of $\ker(f)$. We write $|A|$ for the cardinality of $A$, $I_A$ for the identity map on $A$, and  $A\setminus B$ for the set of elements in $A$ which are not in $B$, where $B$ is any set. Denote by $\Omega(A)$ (resp. $\sym(A)$) the semigroup (resp. group) under composition consisting of all surjective (resp. bijective) maps from $A$ to itself.

%%%%%%%%%%%%%%%%%%%%%%%%%%%%%
%%%%%%%%%%%%%%%%%%%%%%%%%%%%%
%%%%%%%%%%%%%%%%%%%%%%%%%%%%%
%%%%%%%%%%%%%%%%%%%%%%%%%%%%%
\vspace{0.05cm}

Let $f\colon A \to A$ be a map. We write $D(f)$ for the set $A\setminus R(f)$. The \emph{restriction} of $f$ to any subset $B$ of the domain of $f$ is the map $f_{|_B} \colon B \to A$ defined by $x(f_{|_B}) = xf$ for all $x \in B$. If $C$ is a subset of the codomain of $f$ such that
$Af \subseteq C$, the \emph{corestriction} of $f$ to $C$ is the map from $A$ to $C$ that agrees with $f$ on $A$. We write $f_{\upharpoonright_B}$ to denote the corestriction of the map $f_{|_B}$ to $B$ if $Bf \subseteq B$.

%%%%%%%%%%%%%%%%%%%%%%%%%%%%%%
%%%%%%%%%%%%%%%%%%%%%%%%%%%%%%
%%%%%%%%%%%%%%%%%%%%%%%%%%%%%%
%%%%%%%%%%%%%%%%%%%%%%%%%%%%%%
\vspace{0.05cm}

Let $S$ be a semigroup and $a\in S$. If there exists $b\in S$ such that $aba = a$, then $a$ is called a \emph{regular element} in $S$. We write $\reg(S)$ for the set of all regular elements in $S$. If $\reg(S) = S$, then $S$ is said to be a \emph{regular semigroup}. A semigroup $S$ is said to be an \emph{inverse semigroup} if for every $x\in S$, there is a unique $y\in S$ such that  $xyx =x$ and $yxy = y$. Equivalently, we say that $S$ is an \emph{inverse semigroup} if it is a regular semigroup and its idempotents commute (cf. \cite[Theorem 5.1.1]{howie95}). If every element of $S$ belongs to some subgroup of $S$, then $S$ is called a \emph{completely regular semigroup}. We write $E(S)$ for the set of idempotents of $S$.

%We write $\langle A \rangle$ to denote the subsemigroup of $S$ generated by a subset $A$ of $S$
%that consists of all elements of $S$ which can be expressed as finite products of elements in $A$.
%%%%%%%%%%%%%%%%%%%%%%%%%%%%%%
%%%%%%%%%%%%%%%%%%%%%%%%%%%%%%
%%%%%%%%%%%%%%%%%%%%%%%%%%%%%%
%%%%%%%%%%%%%%%%%%%%%%%%%%%%%%

\vspace{0.05cm}
Let $S$ be a semigroup with identity. Denote by $U(S)$ the set of all units in $S$. An element $a\in S$ is called \emph{unit-regular} if there exists $u\in U(S)$ such that $aua = a$. We write $\ureg(S)$ for the set of all unit-regular elements in $S$. If $\ureg(S) = S$, then $S$ is said to be a \emph{unit-regular semigroup}. 

%%%%%%%%%%%%%%%%%%%%%%%%%%%%%%
%%%%%%%%%%%%%%%%%%%%%%%%%%%%%%
%%%%%%%%%%%%%%%%%%%%%%%%%%%%%%
%%%%%%%%%%%%%%%%%%%%%%%%%%%%%%

\vspace{0.05cm}
Let $V$ be a vector space over a field. We denote by $0$ the zero vector of $V$. The subspaces $\{0\}$ and $V$ of $V$ are called \emph{trivial subspaces}. By $\langle T \rangle$, we mean the subspace spanned by a subset $T$ of $V$. For any subspace $U$ of $V$, we write  $\dim(U)$ and $\codim(U)$ to denote the dimensions of $U$ and the quotient space $V/U$, respectively. Let $f\in L(V)$. Denote by $N(f)$ the null space $\{v\in V\colon vf = 0\}$ of $f$. We write $\nul(f)$ (resp. $\corank(f)$) to denote $\dim(N(f))$ (resp. $\dim(V/R(f))$), and $V \approx U$ if vector spaces $V$ and $U$ are isomorphic. We write $\aut(V)$ for the group under composition consisting of all automorphisms of $V$. The notation $\Omega(V)$ can be regarded as a linear version of the semigroup $\Omega(X)$ of surjective transformations of a set $X$.

%%%%%%%%%%%%%%%%%%%%%%%%%%%%%
%%%%%%%%%%%%%%%%%%%%%%%%%%%%%
%%%%%%%%%%%%%%%%%%%%%%%%%%%%%
%%%%%%%%%%%%%%%%%%%%%%%%%%%%%
\vspace{0.05cm}

We also refer the reader to \cite{howie95} and \cite{s-roman07} for undefined definitions and notation of semigroup theory and linear algebra, respectively.

%%%%%%%%%%%%%%%%%%%%%%%%%%%%%
%%%%%%%%%%%%%%%%%%%%%%%%%%%%%
%%%%%%%%%%%%%%%%%%%%%%%%%%%%%
%%%%%%%%%%%%%%%%%%%%%%%%%%%%%
%\newpage
\section{The semigroup $T_{\mathbb{S}(Y)}(X)$} 
In this section, we give a new characterization for the regularity of $T_{\mathbb{S}(Y)}(X)$. Next, we give a new characterization for $T_{\mathbb{S}(Y)}(X)$ to be an inverse semigroup. If $I_Y \in \mathbb{S}(Y)$, we describe unit-regular elements in $T_{\mathbb{S}(Y)}(X)$ and then
determine when $T_{\mathbb{S}(Y)}(X)$ is a unit-regular semigroup. Throughout this section, we assume that $Y$ is a nonempty subset of $X$ and $\mathbb{S}(Y)$ is a subsemigroup of $T(Y)$. 

\vspace{0.05cm}
We begin with the following proposition that gives a sufficient condition for $T_{\mathbb{S}(Y)}(X)$ to be regular.

%%%%%%%%%%%%%%%%%%%%%%%%%%%%%%%
%%%%%%%%%%%%%%%%%%%%%%%%%%%%%%%
%%%%%%%%%%%%%%%%%%%%%%%%%%%%%%%
%%%%%%%%%%%%%%%%%%%%%%%%%%%%%%%

%\begin{remark}\label{subsemigroup-semigroup}
%If $f,g,h\in L_{\mathbb{S}(W)}(V)$. If $h=fg$, then $ h_{\upharpoonright_W}= f_{\upharpoonright_W} g_{\upharpoonright_W}$ where $f_{\upharpoonright_W}, g_{\upharpoonright_W}, h_{\upharpoonright_W}\in \mathbb{S}(W)$.
%\end{remark}
%\begin{remark}\label{reg-subsemgp=reg-semigp}
%Let $S$ be a semigroup with identity and $T$ be a subsemigroup of $S$. Then
%	\begin{enumerate}
	%		\item[\rm(i)] $\reg(T)\subseteq \reg(S)$; 
	%		
	%		\item[\rm(ii)] $\ureg(T)\subseteq \ureg(S)$; and
	%		\item[\rm(iii)] $\ureg(T)=T\cap \ureg(S)$ whenever $U(S)\subseteq T$.
	%	\end{enumerate}
%\end{remark}

%%%%%%%%%%%%%%%%%%%%%%%%%%%%%%%
%%%%%%%%%%%%%%%%%%%%%%%%%%%%%%%
%%%%%%%%%%%%%%%%%%%%%%%%%%%%%%%
%%%%%%%%%%%%%%%%%%%%%%%%%%%%%%%

\begin{proposition}\label{S(W)=group-T(X,Y)=regular}
	If $\mathbb{S}(Y)$ is a subgroup of $\sym(Y)$, then $T_{\mathbb{S}(Y)}(X)$ is regular.
\end{proposition}

\begin{proof}[\textbf{Proof}]
	Let $f\in T_{\mathbb{S}(Y)}(X)$. Then $f_{\upharpoonright_Y} \in \mathbb{S}(Y)$. 
	If $\mathbb{S}(Y)$ is a subgroup of $\sym(Y)$, then it is obvious that $f_{\upharpoonright_Y}\in \reg(\mathbb{S}(Y))$ and $R(f_{\upharpoonright_Y})=R(f)\cap Y$. Therefore $f\in \reg(T_{\mathbb{S}(Y)}(X))$ by \cite[Theorem 2.2 (1)]{koni-sf22}, and hence $T_{\mathbb{S}(Y)}(X)$ is regular.
\end{proof}

%%%%%%%%%%%%%%%%%%%%%%%%%%%%%%%
%%%%%%%%%%%%%%%%%%%%%%%%%%%%%%%
%%%%%%%%%%%%%%%%%%%%%%%%%%%%%%%
%%%%%%%%%%%%%%%%%%%%%%%%%%%%%%%

The following auxiliary lemma is also needed in the sequel.

\begin{lemma}\label{r-subsemigroup=subgroup}
	Let $S$ be a subsemigroup of a group $G$. Then $S$ is a subgroup of $G$ if and only if $S$ is regular.
\end{lemma}

%\begin{proof}[\textbf{Proof}]
%	If $S$ is a subgroup of $G$, then $S$ is clearly regular.
%	
%	\vspace{0.05cm}
%	For the converse, suppose that $S$ is regular. Let $a\in S$. Then there exists $x\in S$ such that $axa = a$. Since $a, axa \in G$ and $G$ is a group, it follows that $ax=e=xa$ by the left and right cancellation laws on $G$, where $e$ denotes the identity of $G$. Therefore the inverse $a^{-1}\in G$ of $a$ is equal to $x$, and so $a^{-1} \in S$. Hence $S$ is a subgroup of $G$ by \cite[Theorem 3.2]{gallian17}. 
%\end{proof}

%%%%%%%%%%%%%%%%%%%%%%%%%%%%%%%
%%%%%%%%%%%%%%%%%%%%%%%%%%%%%%%
%%%%%%%%%%%%%%%%%%%%%%%%%%%%%%%
%%%%%%%%%%%%%%%%%%%%%%%%%%%%%%%

%OK-- 08-03-2022
Janusz Konieczny \cite[Theorem 2.2(2)]{koni-sf22} characterized the regularity of $T_{\mathbb{S}(Y)}(X)$. The following theorem provides a new characterization of the regularity of $T_{\mathbb{S}(Y)}(X)$.

%\begin{theorem}\label{regular-semigroup-T-XY}
%Let $\mathbb{S}(Y)$ be a subsemigroup of $T(Y)$. Then $T_{\mathbb{S}(Y)}(X)$ is regular if and only if one of the following holds:
%\begin{enumerate}
%\item[\rm(i)] $\mathbb{S}(Y)$ is a subgroup of $\sym(Y)$.
%\item[\rm(ii)] $\mathbb{S}(Y)$ is regular, and $Y=X$ or $Y$ is singleton.
%	\end{enumerate}
%\end{theorem}

\begin{theorem}\label{regular-semigroup-T-XY}
Let $\mathbb{S}(Y)$ be a subsemigroup of $T(Y)$. Then $T_{\mathbb{S}(Y)}(X)$ is regular if and only if one of the following holds:
\begin{enumerate}
	\item[\rm(i)] $\mathbb{S}(Y)$ is a subgroup of $\sym(Y)$.
	\item[\rm(ii)] $\mathbb{S}(Y)$ is regular and $Y=X$.
\end{enumerate}	
\end{theorem}

\begin{proof}[\textbf{Proof}]
Suppose that $T_{\mathbb{S}(Y)}(X)$ is regular. By Proposition \ref{S(W)=group-T(X,Y)=regular}, we see that $\rm(i)$ may hold. Let us assume that $\rm(i)$ does not hold.

\vspace{0.05cm}
Since  $T_{\mathbb{S}(Y)}(X)$ is regular, it follows that $\mathbb{S}(Y)$ is regular by \cite[Proposition 2.1]{koni-sf22} and the fact that any homomorphic image of a regular semigroup is regular (cf. \cite[Lemma 2.4.4]{howie95}).

\vspace{0.05cm}
To show $Y =X$, suppose to the contrary that $Y\neq X$. Observe that $|Y|\ge 2$, since $\mathbb{S}(Y)$ is not a subgroup of $\sym(Y)$. Now we claim that $\mathbb{S}(Y)\setminus \Omega(Y)\neq \varnothing$. Suppose to the contrary that $\mathbb{S}(Y)\setminus \Omega(Y)=\varnothing$. Then $\mathbb{S}(Y)\subseteq \Omega(Y)$. Since $\mathbb{S}(Y)$ is regular, we see that $\mathbb{S}(Y)\subseteq \sym(Y)$, and so $\mathbb{S}(Y)$ is a subgroup of $\sym(Y)$ by Lemma \ref{r-subsemigroup=subgroup}. This contradicts the assumption that $\rm(i)$ does not hold. Hence $\mathbb{S}(Y)\setminus \Omega(Y)\neq \varnothing$. Choose $\alpha \in \mathbb{S}(Y)$ such that $\alpha\notin \Omega(Y)$. Fix $y_0\in Y\setminus R(\alpha)$ and  define $f\in T(X)$ by 
\begin{equation*}
	xf=	
	\begin{cases}
		x\alpha   & \text{ if $x\in Y$},\\
		y_0       & \text{ if $x\in X\setminus Y$.}
	\end{cases}
\end{equation*}
Clearly $f\in T_{\mathbb{S}(Y)}(X)$. However, we see that $R(f)\cap Y\neq R(f_{\upharpoonright_Y})$. Therefore $f\notin \reg(T_{\mathbb{S}(Y)}(X))$ by \cite[Theorem 2.2(1)]{koni-sf22}, which contradicts the regularity of $T_{\mathbb{S}(Y)}(X)$. Hence $Y=X$.

\vspace{0.1cm}
For the converse, suppose first that $\rm(i)$ holds. Then $T_{\mathbb{S}(Y)}(X)$ is regular by Proposition \ref{S(W)=group-T(X,Y)=regular}. Next, suppose that $\rm(ii)$ holds. Since $Y = X$, we have $T_{\mathbb{S}(Y)}(X)=\mathbb{S}(Y)$, and so $T_{\mathbb{S}(Y)}(X)$ is regular by hypothesis.
\end{proof}

Janusz Konieczny \cite[Theorem 2.3]{koni-sf22} determined when $T_{\mathbb{S}(Y)}(X)$ is an inverse semigroup. The following theorem provides a new characterization for $T_{\mathbb{S}(Y)}(X)$ to be an inverse semigroup.

\begin{theorem}
Let $\mathbb{S}(Y)$ be a subsemigroup of $T(Y)$. Then $T_{\mathbb{S}(Y)}(X)$ is an inverse semigroup if and only if
\begin{enumerate}
	\item[\rm(i)] $\mathbb{S}(Y)$ is an inverse semigroup;
	\item[\rm(ii)] either $Y=X$ or $|X|=2$.
\end{enumerate}
\end{theorem}

\begin{proof}[\textbf{Proof}]
Suppose that $T_{\mathbb{S}(Y)}(X)$ is an inverse semigroup. Then $\rm(i)$ is true by \cite[Proposition 2.1]{koni-sf22} and the fact that any homomorphic image of an inverse semigroup is an inverse semigroup (cf. \cite[Theorem 5.1.4]{howie95}).

\vspace{0.05cm}
To show $\rm(ii)$, we note that $T_{\mathbb{S}(Y)}(X)=\mathbb{S}(Y)$ if $Y=X$. Therefore $Y=X$ is a possible case. Assume that $Y \neq X$. First, we claim that $|X\setminus Y|=1$. Suppose to the contrary that there exist distinct $x_1,x_2\in X\setminus Y$. Choose $\alpha\in E(\mathbb{S}(Y))$ and define $f,g\in T(X)$ by
\begin{eqnarray*}
	xf=	
	\begin{cases}
		x\alpha   & \text{if $x\in Y$},\\
		x_1       & \text{if $x\in X\setminus Y$}
	\end{cases}
	\qquad \text{ and }\qquad
	xg=	
	\begin{cases}
		x\alpha   & \text{if $x\in Y$},\\
		x_2       & \text{if $x\in X\setminus Y$.}
	\end{cases}	
\end{eqnarray*}
It is routine to verify that $f, g\in E(T_{\mathbb{S}(Y)}(X))$. However, we see that $x_1(fg)\neq x_1(gf)$, and so $fg\neq gf$. This leads to a contradiction, because  $T_{\mathbb{S}(Y)}(X)$ is an inverse semigroup. Hence $|X\setminus Y|=1$. Write $X\setminus Y=\{z\}$.

\vspace{0.05cm}
Next, we claim that $|Y|=1$. Suppose to the contrary that $|Y|\geq 2$. Let $\beta\in E(\mathbb{S}(Y))$. There are two cases to consider.

\vspace{0.05cm}
\noindent\textbf{Case 1:} $|R(\beta)|=1$.
Write $R(\beta)=\{y_1\}$. Since $|Y|\geq 2$, there exists $y_2\in Y$ such that $y_2\neq y_1$. Define $h\in T(X)$ by 
\begin{eqnarray*}
	xh=	
	\begin{cases}
		x\beta   & \text{if $x\in Y$},\\
		y_2      & \text{if $x\in X\setminus Y$}.
	\end{cases}
\end{eqnarray*}
Clearly $h\in T_{\mathbb{S}(Y)}(X)$. However, we see that $R(\beta)\neq Y\cap R(h)$. Therefore $h\notin \reg(T_{\mathbb{S}(Y)}(X))$ by \cite[Theorem 2.2 (1)]{koni-sf22}, and so $T_{\mathbb{S}(Y)}(X)$ is not regular. This leads to a contradiction, because $T_{\mathbb{S}(Y)}(X)$ is an inverse semigroup.

\vspace{0.05cm}
\noindent\textbf{Case 2:} $|R(\beta)|\geq 2$. Choose distinct $y_1, y_2\in R(\beta)$. Define $f,g\in T(X)$ by
\begin{eqnarray*}
	xf=	
	\begin{cases}
		x\beta   & \text{if $x\in Y$},\\
		y_1      & \text{if $x\in X\setminus Y$}
	\end{cases}
	\qquad \text{ and }\qquad
	xg=	
	\begin{cases}
		x\beta   & \text{if $x\in Y$},\\
		y_2      & \text{if $x\in X\setminus Y$}.
	\end{cases}	
\end{eqnarray*}
It is routine to verify that $f, g\in E(T_{\mathbb{S}(Y)}(X))$. However, we see that $z(fg) \neq z(gf)$, and so $fg\neq gf$.
This leads to a contradiction, because  $T_{\mathbb{S}(Y)}(X)$ is an inverse semigroup.

\vspace{0.05cm}
Thus, in either case, we get a contradiction. Hence $|Y|=1$. Since $|X\setminus Y|=1$, we conclude that $|X|=2$.

\vspace{0.1cm}
Conversely, suppose that the given conditions hold. If $Y=X$, then $T_{\mathbb{S}(Y)}(X)=\mathbb{S}(Y)$, and so $T_{\mathbb{S}(Y)}(X)$ is an inverse semigroup by $\rm(i)$. Assume that $Y\neq X$. Then we have $|X|=2$ by (ii), and subsequently $Y$ is singleton. Therefore $\mathbb{S}(Y)$ is a subgroup of $\sym(Y)$, and so $T_{\mathbb{S}(Y)}(X)$ is regular by Theorem \ref{regular-semigroup-T-XY}. Further, since $|X| = 2$ and $|Y| = 1$, it is clear that
$T_{\mathbb{S}(Y)}(X)$ has exactly two idempotents which commute. Hence $T_{\mathbb{S}(Y)}(X)$ is an inverse semigroup.
\end{proof}

%%%%%%%%%%%%%%%%%%%%%%%%%%%%%%%
%%%%%%%%%%%%%%%%%%%%%%%%%%%%%%%
%%%%%%%%%%%%%%%%%%%%%%%%%%%%%%%
%%%%%%%%%%%%%%%%%%%%%%%%%%%%%%%

The following lemma is used in the next two results.

\begin{lemma}\label{map-trans}
Let $f\in T_{\mathbb{S}(Y)}(X)$. Then there exists a subset $T$ of $X$ such that $T$ and $T\cap Y$ are transversals of $\ker(f)$ and $\ker(f_{\upharpoonright_Y})$, respectively.
\end{lemma}

\begin{proof}[\textbf{Proof}]
Observe that $xf^{-1}\cap Y\neq \varnothing$ for all $x\in R(f_{\upharpoonright_Y})$. Therefore for each $x\in R(f_{\upharpoonright_Y})$, fix $x'\in xf^{-1}\cap Y$, and let $T_1=\{x'\colon x\in R(f_{\upharpoonright_Y})\}$. By construction of $T_1$, it is clear that $T_1$ is a transversal of $\ker(f_{\upharpoonright_Y})$. Now we consider  $R(f)\setminus R(f_{\upharpoonright_Y})$. There are two cases.

\vspace{0.05cm}
\noindent\textbf{Case 1:} $R(f)\setminus R(f_{\upharpoonright_Y})=\varnothing$. Take $T=T_1$. Clearly $T\cap Y=T_1$. Thus $T$ and $T\cap Y$ are transversals of $\ker(f)$ and $\ker(f_{\upharpoonright_Y})$, respectively.

\vspace{1mm}
\noindent\textbf{Case 2:} $R(f)\setminus R(f_{\upharpoonright_Y})\neq \varnothing$. Clearly $xf^{-1}\neq \varnothing$ for all $x\in R(f)\setminus R(f_{\upharpoonright_Y})$. Therefore for each $x\in R(f)\setminus R(f_{\upharpoonright_Y})$, fix $x''\in xf^{-1}$, and let $T=T_1\cup\{x''\colon x\in R(f)\setminus R(f_{\upharpoonright_Y})\}$. It is routine to verify that $T$ is a transversal of $\ker(f)$. Now, we see that $T_1=T\cap Y$. Thus $T$ and $T \cap Y$ are transversals of $\ker(f)$ and $\ker(f_{\upharpoonright_Y})$, respectively.
\end{proof}

%%%%%%%%%%%%%%%%%%%%%%%%%%%%%%%
%%%%%%%%%%%%%%%%%%%%%%%%%%%%%%%
%%%%%%%%%%%%%%%%%%%%%%%%%%%%%%%
%%%%%%%%%%%%%%%%%%%%%%%%%%%%%%%

If $I_Y \in \mathbb{S}(Y)$, then it is clear that $I_X \in T_{\mathbb{S}(Y)}(X)$. The following theorem characterizes unit-regular elements in $T_{\mathbb{S}(Y)}(X)$ when $I_Y \in \mathbb{S}(Y)$.

\begin{theorem}\label{unit-regular elment T-XY}
Let $\mathbb{S}(Y)$ be a subsemigroup of $T(Y)$ such that $I_Y\in \mathbb{S}(Y)$, and let $f\in T_{\mathbb{S}(Y)}(X)$. Then $f\in \ureg(T_{\mathbb{S}(Y)}(X))$ if and only if 
\begin{enumerate}
	\item[\rm(i)] $f_{\upharpoonright_Y}\in \ureg(\mathbb{S}(Y))$;
	\item[\rm(ii)] $R(f)\cap Y=R(f_{\upharpoonright_Y})$;
	\item[\rm(iii)] $|C(f)\setminus C(f_{\upharpoonright_Y})|=|D(f)\setminus D(f_{\upharpoonright_Y})|$, where $C(f)=X\setminus T_f$ and $C(f_{\upharpoonright_Y})=Y\setminus T_{(f_{\upharpoonright_Y})}$ for some transversals $T_f$ and $T_{(f_{\upharpoonright_Y})}$ of $\ker(f)$ and $\ker(f_{\upharpoonright_Y})$, respectively, such that $Y\cap T_f = T_{(f_{\upharpoonright_Y})}$.
\end{enumerate}
\end{theorem}

\begin{proof}[\textbf{Proof}]
Suppose that $f\in \ureg(T_{\mathbb{S}(Y)}(X))$. Then there exists $g\in U(T_{\mathbb{S}(Y)}(X))$ such that $fgf = f$. It follows that $g_{\upharpoonright_Y}\in U(\mathbb{S}(Y))$ and $f_{\upharpoonright_Y} g_{\upharpoonright_Y} f_{\upharpoonright_Y} = f_{\upharpoonright_Y}$, and so $f_{\upharpoonright_Y}\in \ureg(\mathbb{S}(Y))$. Thus $\rm(i)$ holds.	

\vspace{0.05cm}
To show (ii) and (iii), we notice that $T_{T(Y)}(X)$ is the semigroup $\overline{T}(X, Y) = \{f\in T(X)\colon Yf \subseteq Y\}$ under composition. Since $f\in \ureg(T_{\mathbb{S}(Y)}(X))$ and $T_{\mathbb{S}(Y)}(X) \subseteq T_{T(Y)}(X)$, we have $f\in \ureg(\overline{T}(X, Y))$. Hence $\rm(ii)$ and $\rm(iii)$ are directly followed by \cite[Theorem 4.2]{shubh-lma22}.

\vspace{0.1cm}
Conversely, suppose that $f$ satisfies \rm(i)--\rm(iii). By $\rm(i)$, there exists $g_0\in U(\mathbb{S}(Y))$ such that $f_{\upharpoonright_Y} g_0 f_{\upharpoonright_Y}=f_{\upharpoonright_Y}$. 

\vspace{0.05cm}
Now, we observe that $R(f)\setminus Y\neq\varnothing$ if and only if $T_f\setminus Y\neq \varnothing$. If $R(f)\setminus Y\neq \varnothing$, then we show that $g_1\colon R(f)\setminus Y\to T_f\setminus Y$ defined by $xg_1=x'$, where $x'\in T_f\cap xf^{-1}$, is bijective. Clearly $g_1$ is injective, since $x_1f^{-1}\cap x_2f^{-1}=\varnothing$ for distinct $x_1,x_2\in R(f)$. To show $g_1$ is surjective, let $x\in T_f\setminus Y$ and write $z = xf$. Certainly $z\in R(f)$. By (ii), note that $R(f)\cap Y=R(f_{\upharpoonright_Y})$. Therefore, since $T_{(f_{\upharpoonright_Y})}=Y\cap T_f$, we have $z\in R(f)\setminus Y$. Also, we see that $zg_1=x$, and so $g_1$ is surjective.

\vspace{0.05cm}
By $\rm(iii)$, there is a bijection $g_2\colon D(f)\setminus D(f_{\upharpoonright_Y})\to C(f)\setminus C(f_{\upharpoonright_Y})$. Notice that $X=Y\cup \big(R(f)\setminus Y\big)\cup \big(D(f)\setminus D(f_{\upharpoonright_Y})\big)$ and also $X=Y\cup \big(T_f\setminus Y\big)\cup \big(C(f)\setminus C(f_{\upharpoonright_Y})\big)$. Define $g\in T(X)$ by
\begin{equation*}
	xg=
	\begin{cases}
		xg_0    & \text{if $x\in Y$},\\
		xg_1    & \text{if $x\in R(f)\setminus Y$},\\
		xg_2    & \text{if $x\in D(f)\setminus D(f_{\upharpoonright_Y})$.}\\
	\end{cases}
\end{equation*}
It is routine to verify that $g$ is bijective, $g\in T_{\mathbb{S}(Y)}(X)$, and $fgf = f$. Hence $f\in \ureg(T_{\mathbb{S}(Y)}(X))$.
\end{proof}

%%%%%%%%%%%%%%%%%%%%%%%%%%%%%%%
%%%%%%%%%%%%%%%%%%%%%%%%%%%%%%%
%%%%%%%%%%%%%%%%%%%%%%%%%%%%%%%
%%%%%%%%%%%%%%%%%%%%%%%%%%%%%%%

The following proposition gives a sufficient condition for $T_{\mathbb{S}(Y)}(X)$ to be unit-regular.

\begin{proposition}\label{subgroup-unit-regular}
If $\mathbb{S}(Y)$ is a subgroup of $\sym(Y)$ and $X\setminus Y$ is finite, then $T_{\mathbb{S}(Y)}(X)$ is unit-regular.
\end{proposition}

\begin{proof}[\textbf{Proof}]
Let $f\in T_{\mathbb{S}(Y)}(X)$. Then $f_{\upharpoonright_Y} \in \mathbb{S}(Y)$. Since $\mathbb{S}(Y)$ is a subgroup of $\sym(Y)$, it is clear that $f_{\upharpoonright_Y}\in \ureg(\mathbb{S}(Y))$ and $R(f)\cap Y=R(f_{\upharpoonright_Y})$. Thus $f$ satisfies the conditions (i) and (ii) of Theorem \ref{unit-regular elment T-XY}. 

\vspace{0.05cm}
By Lemma \ref{map-trans}, let $T_f$ and $T_{(f_{\upharpoonright_Y})}$ be transversals of $\ker(f)$ and $\ker(f_{\upharpoonright_Y})$, respectively, such that $T_f\cap Y=T_{(f_{\upharpoonright_Y})}$. Since $f_{\upharpoonright_Y}\in \mathbb{S}(Y)$ and $\mathbb{S}(Y)$ is a subgroup of $\sym(Y)$, we have $T_{(f_{\upharpoonright_Y})}= Y=R(f_{\upharpoonright_Y})$. Therefore $Y\subseteq T_f$ and $Y\subseteq R(f)$. Since $X\setminus Y$ is finite, it follows that both $T_f\setminus Y$ and $R(f)\setminus Y$ are finite.

\vspace{0.05cm}
Claim that $|T_f\setminus Y|=|R(f)\setminus Y|$. Notice that $R(f)\setminus Y\neq\varnothing$ if and only if $T_f\setminus Y\neq \varnothing$. If $R(f)\setminus Y\neq \varnothing$, then we show that $\alpha\colon R(f)\setminus Y\to T_f\setminus Y$ defined by $x\alpha=x'$, where $x'\in T_f\cap xf^{-1}$, is bijective. Clearly $\alpha$ is injective, since $x_1f^{-1}\cap x_2f^{-1}=\varnothing$ for distinct $x_1,x_2\in R(f)$. To show $\alpha$ is surjective, let $x\in T_f\setminus Y$  and write $z =xf$. Certainly $z\in R(f)$. 
Since $R(f)\cap Y=R(f_{\upharpoonright_Y})$ and $T_{(f_{\upharpoonright_Y})}=Y\cap T_f$, it follows that $z\in R(f)\setminus Y$. Also, we see that $z\alpha=x$, and so $\alpha$ is surjective. Thus $|T_f\setminus Y|=|R(f)\setminus Y|$.

\vspace{0.05cm}
Observe that $D(f)= (X\setminus Y)\setminus (R(f)\setminus Y)$. Write $C(f) = X\setminus T_f$ and $C(f_{\upharpoonright_Y})=Y\setminus T_{(f_{\upharpoonright_Y})}$. It is clear that $C(f) = (X\setminus Y)\setminus (T_f\setminus Y)$. Since $X\setminus Y$, $T_f\setminus Y$, and $R(f)\setminus Y$ are finite, and $|T_f\setminus Y|=|R(f)\setminus Y|$, we obtain
\[|C(f)|=|X\setminus Y|-|T_f\setminus Y|=|X\setminus Y|-|R(f)\setminus Y|=|(X\setminus Y)\setminus \big(R(f)\setminus Y\big)| = D(f).\]
Since $f_{\upharpoonright_Y}$ is a bijection on $Y$, we have $C(f_{\upharpoonright_Y})=\varnothing$ and $D(f_{\upharpoonright_Y})=\varnothing$. Therefore $|C(f)\setminus C(f_{\upharpoonright_Y})|=|C(f)| = |D(f)| = |D(f)\setminus D(f_{\upharpoonright_Y})|$. Thus $f$ also satisfies the condition (iii) of Theorem \ref{unit-regular elment T-XY}, and so $f\in \ureg(T_{\mathbb{S}(Y)}(X))$ by Theorem \ref{unit-regular elment T-XY}. Hence, since $f$ is arbitrary, we conclude that $T_{\mathbb{S}(Y)}(X)$ is a unit-regular semigroup.

\end{proof}

%%%%%%%%%%%%%%%%%%%%%%%%%%%%%%%
%%%%%%%%%%%%%%%%%%%%%%%%%%%%%%%
%%%%%%%%%%%%%%%%%%%%%%%%%%%%%%%
%%%%%%%%%%%%%%%%%%%%%%%%%%%%%%%
%OK--08-03-2022
The next theorem gives a necessary and sufficient condition for $T_{\mathbb{S}(Y)}(X)$ to be unit-regular if $I_Y\in \mathbb{S}(Y)$. 
\begin{theorem}
Let $\mathbb{S}(Y)$ be a subsemigroup of $T(Y)$ such that $I_Y\in \mathbb{S}(Y) $. Then $T_{\mathbb{S}(Y)}(X)$ is unit-regular if and only if one of the following holds:
\begin{enumerate}
	\item[\rm(i)] $\mathbb{S}(Y)$ is a subgroup of $\sym(Y)$ and $X\setminus Y$ is finite.
	\item[\rm(ii)] $\mathbb{S}(Y)$ is unit-regular and $Y=X$.
\end{enumerate}
\end{theorem}

\begin{proof}[\textbf{Proof}]
Suppose that $T_{\mathbb{S}(Y)}(X)$ is unit-regular. By Proposition \ref{subgroup-unit-regular}, we see that $\rm(i)$ may hold. Let us assume that $\rm(i)$ does not hold.

\vspace{0.05cm}	
Since $T_{\mathbb{S}(Y)}(X)$ is unit-regular, it follows that $\mathbb{S}(Y)$ is unit-regular by \cite[Propsition 2.1]{koni-sf22} and the fact that any homomorphic image of a unit-regular semigroup is unit-regular (cf. \cite[Proposition 2.7]{hick97}).

\vspace{0.05cm}
To show $Y = X$, suppose to the contrary that $Y\neq X$. First, we claim that
$\mathbb{S}(Y)$ is not a subgroup of $\sym(Y)$. Suppose to the contrary that $\mathbb{S}(Y)$ is a subgroup of $\sym(Y)$. Then we next claim that $X\setminus Y$ is finite.  Suppose to the contrary that $X\setminus Y$ is infinite. Then there exists $\alpha\colon X\setminus Y\to X\setminus Y$ which is injective but not surjective. Define $f\in T(X)$ by
\begin{align*}
	xf=
	\begin{cases}
		x       & \text{if $x\in Y$},\\
		x\alpha & \text{if $x\in X\setminus Y$.}	
	\end{cases}
\end{align*}
Clearly $f_{\upharpoonright_Y} = I_Y$ and $f\in T_{\mathbb{S}(Y)}(X)$. Moreover, we see that $f$ is injective. It follows that $Y$ and $X$ are only transversals of $\ker(f_{\upharpoonright_Y})$ and $\ker(f)$, respectively. Therefore $C(f) = X\setminus T_f = \varnothing$ and $C(f_{\upharpoonright_Y}) = Y\setminus T_{(f_{\upharpoonright_Y})} = \varnothing$, and so $|C(f)\setminus C(f_{\upharpoonright_Y})| = 0$. Notice that $f_{\upharpoonright_Y}$ is surjective but $f$ is not surjective. Therefore $D(f_{\upharpoonright_Y}) = \varnothing$ but $D(f)\neq \varnothing$, and so $|D(f)\setminus D(f_{\upharpoonright_Y})| \neq 0$.
Thus $f$ does not satisfy Theorem \ref{unit-regular elment T-XY}(iii), and hence $f\notin \ureg(T_{\mathbb{S}(Y)}(X))$ by Theorem \ref{unit-regular elment T-XY}, which contradicts the unit-regularity of $T_{\mathbb{S}(Y)}(X)$. Thus $X\setminus Y$ is finite, and subsequently $\rm(i)$ holds. This contradicts the assumption that $\rm(i)$ does not hold. Therefore $\mathbb{S}(Y)$ is not a subgroup of $\sym(Y)$. Hence, since $T_{\mathbb{S}(Y)}(X)$ is regular, we have $Y=X$ by Theorem \ref{regular-semigroup-T-XY}.

\vspace{0.05cm}
For the converse, suppose first that $\rm(i)$ holds. Then $T_{\mathbb{S}(Y)}(X)$ is unit-regular by Proposition \ref{subgroup-unit-regular}. Next, suppose that $\rm(ii)$ holds. Since $Y = X$, we have $T_{\mathbb{S}(Y)}(X)=\mathbb{S}(X)$, and so $T_{\mathbb{S}(Y)}(X)$ is a unit-regular semigroup by hypothesis.
\end{proof}

\section{The semigroup $L_{\mathbb{S}(W)}(V)$} 

In this section, we characterize regular elements in $L_{\mathbb{S}(W)}(V)$ and then determine when $L_{\mathbb{S}(W)}(V)$ is a regular semigroup. If $I_W \in \mathbb{S}(W)$, then we characterize unit-regular elements in $L_{\mathbb{S}(W)}(V)$ and then determine when $L_{\mathbb{S}(W)}(V)$ is a unit-regular semigroup. We also determine when $L_{\mathbb{S}(W)}(V)$ is a completely regular semigroup, and when $L_{\mathbb{S}(W)}(V)$ is an inverse semigroup. Throughout this section, we assume that $W$ is a subspace of $V$ and $\mathbb{S}(W)$ is a subsemigroup of $L(W)$. We begin with the following proposition.

%As a consequence, we determine when $L(V)$ is an inverse semigroup, and when $L(V)$ is a completely semigroup.
%%%%%%%%%%%%%%%%%%%%%%%%%%%%%%%
%%%%%%%%%%%%%%%%%%%%%%%%%%%%%%%
%%%%%%%%%%%%%%%%%%%%%%%%%%%%%%%
%%%%%%%%%%%%%%%%%%%%%%%%%%%%%%%

\begin{proposition}\label{epimorphism-L-VW}
Let $\mathbb{S}(W)$ be a subsemigroup of $L(W)$. Then $\varphi \colon L_{\mathbb{S}(W)}(V)\to \mathbb{S}(W)$ defined by $f\varphi= f_{\upharpoonright_W}$ is an epimorphism. Moreover, if $I_V\in  L_{\mathbb{S}(W)}(V)$, then $(I_V)\varphi=I_W$.
\end{proposition}

\begin{proof}[\textbf{Proof}]
Observe that $(fg)_{\upharpoonright_W} = f_{\upharpoonright_W} g_{\upharpoonright_W}$ for all $f,g\in L_{\mathbb{S}(W)}(V)$. Clearly $\varphi \colon L_{\mathbb{S}(W)}(V)\to \mathbb{S}(W)$ is a homomorphism, since 
\[(fg)\varphi=(fg)_{\upharpoonright_W} = f_{\upharpoonright_W} g_{\upharpoonright_W}=(f\varphi)(g\varphi)\]
for all $f,g\in L_{\mathbb{S}(W)}(V)$. To show $\varphi$ is surjective, let $\alpha\in \mathbb{S}(W)$. Extend $\alpha$ to a linear operator, say $f$ on $V$. Then we see that $f\varphi=f_{\upharpoonright_W}=\alpha$, and so $\varphi$ is surjective. Thus $\varphi$ is an epimorphism. Further, if $I_V\in  L_{\mathbb{S}(W)}(V)$, then it is clear that $(I_V)\varphi=I_W$.		
\end{proof}	

%%%%%%%%%%%%%%%%%%%%%%%%%%%%%%%
%%%%%%%%%%%%%%%%%%%%%%%%%%%%%%%
%%%%%%%%%%%%%%%%%%%%%%%%%%%%%%%
%%%%%%%%%%%%%%%%%%%%%%%%%%%%%%%

%\begin{lemma}\label{restrict-reg-ureg}
%Let $\mathbb{S}(W)$ be a subsemigroup of $L(W)$ and $f\in L_{\mathbb{S}(W)}(V)$.  
%	\begin{enumerate}
%		\item[\rm(i)] If $f\in \reg(L_{\mathbb{S}(W)}(V))$, then $f_{\upharpoonright_W}\in \reg(\mathbb{S}(W))$.
%		\item[\rm(ii)] If $f\in \ureg(L_{\mathbb{S}(W)}(V))$, then $f_{\upharpoonright_W}\in \ureg(\mathbb{S}(W))$.
%	\end{enumerate}
%\end{lemma}
%
%
%\begin{proof}\
%\begin{enumerate}
%\item[\rm(i)] If $f\in \reg(L_{\mathbb{S}(W)}(V))$, then there exists $g\in L_{\mathbb{S}(W)}(V)$ such that $fgf=f$. This gives $g_{\upharpoonright_W}\in \mathbb{S}(W)$ and  $f_{\upharpoonright_W} g_{\upharpoonright_W} f_{\upharpoonright_W} = f_{\upharpoonright_W}$, and so $f_{\upharpoonright_W}\in \reg(\mathbb{S}(W))$.
%	
%\vspace{0.1cm}
%\item[\rm(ii)] If $f\in \ureg(L_{\mathbb{S}(W)}(V))$, then there exists $g\in U(L_{\mathbb{S}(W)}(V))$ such that $fgf=f$. This gives $g_{\upharpoonright_W}\in U(\mathbb{S}(W))$ and  $f_{\upharpoonright_W} g_{\upharpoonright_W} f_{\upharpoonright_W} = f_{\upharpoonright_W}$, and so $f_{\upharpoonright_W}\in \ureg(\mathbb{S}(W))$.
%
%\end{enumerate}	
%\end{proof}

%%%%%%%%%%%%%%%%%%%%%%%%%%%%%%%
%%%%%%%%%%%%%%%%%%%%%%%%%%%%%%%
%%%%%%%%%%%%%%%%%%%%%%%%%%%%%%%
%%%%%%%%%%%%%%%%%%%%%%%%%%%%%%%
The next theorem gives a characterization of regular elements in $L_{\mathbb{S}(W)}(V)$.

\begin{theorem}\label{regular-element-L-VW}
Let $\mathbb{S}(W)$ be a subsemigroup of $L(W)$ and $f\in L_{\mathbb{S}(W)}(V)$. Then $f\in \reg(L_{\mathbb{S}(W)}(V))$ if and only if 
\begin{enumerate}
	\item[\rm(i)] $f_{\upharpoonright_W}\in \reg(\mathbb{S}(W))$;
	\item[\rm(ii)] $R(f)\cap W=R(f_{\upharpoonright_W})$.
\end{enumerate}
\end{theorem}

\begin{proof}[\textbf{Proof}]
Suppose that $f\in \reg(L_{\mathbb{S}(W)}(V))$. Then there exists $g\in L_{\mathbb{S}(W)}(V)$ such that $fgf=f$. It follows that $g_{\upharpoonright_W}\in \mathbb{S}(W)$ and  $f_{\upharpoonright_W} g_{\upharpoonright_W} f_{\upharpoonright_W} = f_{\upharpoonright_W}$, and so $f_{\upharpoonright_W}\in \reg(\mathbb{S}(W))$. Thus (i) holds. To show (ii), note that $L_{\mathbb{S}(W)}(V)\subseteq L_{L(W)}(V)$ and $L_{L(W)}(V) = \overline{L}(V, W)$. Since $f\in \reg(L_{\mathbb{S}(W)}(V))$, we have $f\in \reg(\overline{L}(V, W))$, and hence (ii) holds by \cite[Proposition 3.1]{nenth-ijmms06}.

\vspace{0.05cm}
Conversely, suppose that $f$ satisfies the given conditions. By $\rm(i)$, there exists $\alpha\in \mathbb{S}(W)$ such that $f_{\upharpoonright_W}\alpha f_{\upharpoonright_W} =f_{\upharpoonright_W}$. Now, let $B_1$ be a basis for $R(f)\cap W$. Extend $B_1$ to bases $B_1\cup B_2$ and $B_1\cup B_3$ for $W$ and $R(f)$, respectively, where $B_2\subseteq W\setminus \langle B_1\rangle$ and $B_3\subseteq R(f)\setminus \langle B_1\rangle$. Then $B_1\cup B_2\cup B_3$ is a basis for $W +R(f)$. Extend $B_1\cup B_2\cup B_3$ to a basis $B_1\cup B_2\cup B_3\cup B_4$ for $V$, where $B_4 \subseteq V \setminus (W + R(f))$. Since $B_3\subseteq R(f)$, we have $vf^{-1}\neq \varnothing$ for all $v\in B_3$. Therefore for each $v\in B_3$, choose $v'\in vf^{-1}$. Define $h\in L(V)$ by 

\begin{equation*}
	vh=	
	\begin{cases}
		v\alpha   & \text{ if $v\in B_1\cup B_2$},\\
		v'        & \text{ if $v\in B_3$},\\
		0         & \text{ if $v\in B_4$}.
	\end{cases}
\end{equation*}
It is routine to verify that $h\in L_{\mathbb{S}(W)}(V)$ and $fhf=f$. Hence $f\in \reg(L_{\mathbb{S}(W)}(V))$.
\end{proof}

%%%%%%%%%%%%%%%%%%%%%%%%%%%%%%%
%%%%%%%%%%%%%%%%%%%%%%%%%%%%%%%
%%%%%%%%%%%%%%%%%%%%%%%%%%%%%%%
%%%%%%%%%%%%%%%%%%%%%%%%%%%%%%%

The following proposition gives a sufficient condition for $L_{\mathbb{S}(W)}(V)$ to be regular.

\begin{proposition}\label{S(W)=group-L(V,W)=regular}
If  $\mathbb{S}(W)$ is a subgroup of $\aut(W)$, then $L_{\mathbb{S}(W)}(V)$ is regular.
\end{proposition}

\begin{proof}[\textbf{Proof}]
Let $f\in L_{\mathbb{S}(W)}(V)$. Then $f_{\upharpoonright_W}\in \mathbb{S}(W)$. If $\mathbb{S}(W)$ is a subgroup of $\aut(W)$, then we certainly have $f_{\upharpoonright_W} \in \reg(\mathbb{S}(W))$ and $R(f_{\upharpoonright_W})=W=R(f)\cap W$. Therefore  $f\in \reg(L_{\mathbb{S}(W)}(V))$ by Theorem \ref{regular-element-L-VW}. Hence $L_{\mathbb{S}(W)}(V)$ is regular.		
\end{proof}
%%%%%%%%%%%%%%%%%%%%%%%%%%%%%%%
%%%%%%%%%%%%%%%%%%%%%%%%%%%%%%%
%%%%%%%%%%%%%%%%%%%%%%%%%%%%%%%
%%%%%%%%%%%%%%%%%%%%%%%%%%%%%%%

The following theorem gives a necessary and sufficient condition for $L_{\mathbb{S}(W)}(V)$ to be regular.

\begin{theorem}\label{regular semigroup}
Let $\mathbb{S}(W)$ be a subsemigroup of $L(W)$. Then $L_{\mathbb{S}(W)}(V)$ is regular if and only if one of the following holds:
\begin{enumerate}
	\item[\rm(i)] $\mathbb{S}(W)$ is a subgroup of $\aut(W)$.
	\item[\rm(ii)] $\mathbb{S}(W)$ is regular and $W = V$.
\end{enumerate}
\end{theorem}

\begin{proof}[\textbf{Proof}]
Suppose that $L_{\mathbb{S}(W)}(V)$ is regular. By Proposition \ref{S(W)=group-L(V,W)=regular}, we see that $\rm(i)$ may hold. Let us assume that $\rm(i)$ does not hold.

\vspace{0.05cm}
Since $L_{\mathbb{S}(W)}(V)$ is regular, it follows that $\mathbb{S}(W)$ is regular by Proposition \ref{epimorphism-L-VW} and the fact that any homomorphic image of a regular semigroup is regular (cf. \cite[Lemma 2.4.4]{howie95}).

\vspace{0.05cm}
To show $W = V$, suppose to the contrary that $W\neq V$. Observe that $W\neq\{0\}$, since $\mathbb{S}(W)$ is not a subgroup of $\aut(W)$. Now we claim that $\mathbb{S}(W)\setminus \Omega(W)\neq \varnothing$. Suppose to the contrary that $\mathbb{S}(W)\setminus \Omega(W)=\varnothing$. Then $\mathbb{S}(W)\subseteq \Omega(W)$. Since $\mathbb{S}(W)$ is regular, we simply get $\mathbb{S}(W)\subseteq \aut(W)$. It follows that $\mathbb{S}(W)$ is a subgroup of $\aut(W)$ by Lemma \ref{r-subsemigroup=subgroup}, a contradiction to the assumption that $\rm(i)$ does not hold. Hence $\mathbb{S}(W)\setminus \Omega(W)\neq \varnothing$. Thus we choose $\alpha \in \mathbb{S}(W)$ such that $\alpha \notin \Omega(W)$. Clearly $W\setminus R(\alpha)\neq \varnothing$. Let $B_1$ be a basis for $R(\alpha)$. Extend $B_1$ to a basis $B_1\cup B_2$ for $W$, where $B_2\subseteq W\setminus R(\alpha)$. Next, extend $B_1\cup B_2$ to a basis $B_1\cup B_2\cup B_3$ for $V$, where $B_3\subseteq V\setminus W$. Fix $v_0\in B_2$, and define $f\in L(V)$ by
\begin{equation*}
	vf=	
	\begin{cases}
		v\alpha   & \text{ if $v\in B_1\cup B_2$},\\
		v_0       & \text{ if $v\in B_3$.}
	\end{cases}
\end{equation*}
Clearly $f\in L_{\mathbb{S}(W)}(V)$. However, we see that $R(f_{\upharpoonright_W})=R(\alpha)$ and $R(f)\cap W=\langle R(\alpha)\cup \{v_0\}\rangle$, and so $R(f)\cap W\neq R(f_{\upharpoonright_W})$. Therefore $f\notin \reg(L_{\mathbb{S}(W)}(V))$ by Theorem \ref{regular-element-L-VW}, which  contradicts the regularity of $L_{\mathbb{S}(W)}(V)$. Hence $W=V$.

\vspace{0.05cm}
For the converse, suppose first that $\rm(i)$ holds. Then $L_{\mathbb{S}(W)}(V)$ is regular by Proposition \ref{S(W)=group-L(V,W)=regular}. Next, we suppose that $\rm(ii)$ holds. Since $W=V$, we have $L_{\mathbb{S}(W)}(V)=\mathbb{S}(V)$, and hence $L_{\mathbb{S}(W)}(V)$ is regular by hypothesis.
\end{proof}

The following lemma is used in the next two results.

\begin{lemma}\label{subspace-transversal}
Let $f\in L_{\mathbb{S}(W)}(V)$. Then there exists a subspace $U$ of $V$ such that $U$ and $U \cap W$ are transversals of $\ker(f)$ and $\ker(f_{\upharpoonright_W})$, respectively.
\end{lemma}

\begin{proof}[\textbf{Proof}]
Let $B_1$ be a basis for $R(f_{\upharpoonright_W})$. If $B_1\neq\varnothing$, then we see that $uf^{-1}\cap W\neq \varnothing$ for all $u\in B_1$. Therefore for each $u\in B_1$, fix $u'\in uf^{-1}\cap W$ and let $C_1=\{u'\colon u\in B_1\}$. 
If $B_1=\varnothing$, then we take $C_1=\varnothing$. 

\vspace{0.05cm}
In either case, we claim that $U_1=\langle C_1\rangle$ is a transversal of $\ker(f_{\upharpoonright_W})$. If $C_1=\varnothing$, then $U_1=\{0\}$, and so we are done. Assume that $C_1\neq\varnothing$ and let $w\in R(f_{\upharpoonright_W})$. Since $B_1$ is a basis for $R(f_{\upharpoonright_W})$, we have $w=c_1u_1+\cdots +c_mu_m$ for some $u_1,\ldots ,u_m\in B_1$, where $m\geq 0$. Consider $w'= c_1u_1'+\cdots +c_mu_m' \in U_1$, where $u_1',\ldots, u_m'\in C_1$. Observe that $w'f=w$, and so $w'\in wf^{-1}\cap U_1$. Recall that $C_1$ is a basis for $U_1$. Therefore, by construction of $C_1$, we see that $wf^{-1}\cap U_1=\{w'\}$. Hence, since $w\in R(f_{\upharpoonright_W})$ is arbitrary, the subspace $U_1$ of $V$ is a transversal of $\ker{(f_{\upharpoonright_W})}$. Now we consider $R(f)\setminus R(f_{\upharpoonright_W})$. There are two cases.

\vspace{0.05cm}
\noindent\textbf{Case 1:} Suppose $R(f)\setminus R(f_{\upharpoonright_W})=\varnothing$. Take $U=U_1$. Clearly $U\cap W=U_1$. Thus $U$ and $U\cap W$ are transversals of $\ker(f)$ and $\ker(f_{\upharpoonright_W})$, respectively.

\vspace{0.05cm}
\noindent\textbf{Case 2:} Suppose $R(f)\setminus R(f_{\upharpoonright_W})\neq\varnothing$. Extend $B_1$ to a basis $B_1\cup B_2$ for $R(f)$, where $B_2\subseteq R(f)\setminus R(f_{\upharpoonright_W})$. Note that $B_2\neq \varnothing$ and $vf^{-1}\neq \varnothing$ for all $v\in B_2$. Therefore for each $v\in B_2$, fix $\bar{v}\in vf^{-1}$. Write $C=C_1\cup \{\bar{v} \in vf^{-1}\mid v\in B_2\}$ and let $U = \langle C\rangle$. Clearly $U\cap W=U_1$, and so $U\cap W$ is a transversal of  $\ker(f_{\upharpoonright_W})$. Now it remains to show that $U$ is a transversal of $\ker(f)$. For this, let $w\in R(f)$. Since $B_1\cup B_2$ is a basis for $R(f)$, we have $w=c_1u_1+\cdots +c_mu_m+d_1v_1+\cdots +d_nv_n$ for some $u_1,\ldots ,u_m\in B_1$ and $v_1,\ldots ,v_n\in B_2$, where $m,n\geq 0$. Let $w'= c_1u_1'+\cdots +c_mu_m'+d_1\bar{v}_1+\cdots +d_n\bar{v}_n \in U$, where $u_1',\ldots, u_m',\bar{v}_1,\ldots,\bar{v}_n\in C$. Observe that $w'f=w$, and so $w'\in wf^{-1}\cap U$. Recall that $C$ is a basis for $U$. Therefore, by construction of $C$, we see that $wf^{-1}\cap U=\{w'\}$. Hence, since $w\in R(f)$ is arbitrary, the subspace $U$ of $V$ is a transversal of $\ker{(f)}$.
\end{proof}

%%%%%%%%%%%%%%%%%%%%%%%%%%%%%%%
%%%%%%%%%%%%%%%%%%%%%%%%%%%%%%%
%%%%%%%%%%%%%%%%%%%%%%%%%%%%%%%
%%%%%%%%%%%%%%%%%%%%%%%%%%%%%%%
%\newpage
Observe that $I_V$ is not an element of $L_{\mathbb{S}(W)}(V)$, in general. However, if $I_W \in \mathbb{S}(W)$, then it is clear that $I_V \in L_{\mathbb{S}(W)}(V)$. The following theorem provides a characterization of unit-regular elements in $L_{\mathbb{S}(W)}(V)$ when $I_W \in \mathbb{S}(W)$. 

\begin{theorem}\label{unit-regular elment}
Let $\mathbb{S}(W)$ be a subsemigroup of $L(W)$ such that $I_W\in \mathbb{S}(W)$, and let $f\in L_{\mathbb{S}(W)}(V)$. Then $f\in \ureg(L_{\mathbb{S}(W)}(V))$ if and only if 
\begin{enumerate}
	\item[\rm(i)] $f_{\upharpoonright_W}\in \ureg(\mathbb{S}(W))$;
	\item[\rm(ii)] $R(f)\cap W=R(f_{\upharpoonright_W})$;
	\item[\rm(iii)] $\codim(W+T_f)=\codim(W+R(f))$ for some subspace $T_f$ of $V$ such that $T_f$ and $W\cap T_f$ are transversals of $\ker(f)$ and $\ker(f_{\upharpoonright_W})$, respectively.	
\end{enumerate}
\end{theorem}

\begin{proof}[\textbf{Proof}]
Suppose that $f\in \ureg(L_{\mathbb{S}(W)}(V))$. Then there exists $g\in U(L_{\mathbb{S}(W)}(V))$ such that $fgf = f$. This simply gives $g_{\upharpoonright_W}\in U(\mathbb{S}(W))$ and $f_{\upharpoonright_W} g_{\upharpoonright_W} f_{\upharpoonright_W} = f_{\upharpoonright_W}$, and so $f_{\upharpoonright_W}\in \ureg(\mathbb{S}(W))$. Thus $\rm(i)$ holds.

\vspace{0.05cm}
To show (ii) and (iii), note that $L_{L(W)}(V) = \overline{L}(V, W)$. Since $f\in \ureg(L_{\mathbb{S}(W)}(V))$ and $L_{\mathbb{S}(W)}(V) \subseteq L_{L(W)}(V)$, we have $f\in \ureg(\overline{L}(V, W))$. Hence $\rm(ii)$ and $\rm(iii)$ are directly followed by \cite[Theorem 5.6]{shubh-lma22}.

\vspace{0.1cm}
Conversely, suppose that $f$ satisfies \rm(i)--\rm(iii). By $\rm(i)$, there exists $g_0\in U(\mathbb{S}(W))$ such that $f_{\upharpoonright_W} g_0f_{\upharpoonright_W}=f_{\upharpoonright_W}$. Now, let $B_1$ be a basis for $R(f)\cap W$. Extend $B_1$ to bases $B_1\cup B_2$ and $B_1\cup B_3$ for $W$ and $R(f)$, respectively, where $B_2\subseteq W\setminus \langle B_1\rangle$ and $B_3\subseteq R(f)\setminus \langle B_1\rangle$. Then $B_1\cup B_2\cup B_3$ is a basis for $W+R(f)$. Extend $B_1\cup B_2\cup B_3$ to a basis $B:= B_1\cup B_2\cup B_3\cup B_4$ for $V$, where $B_4 \subseteq V \setminus \big(W+R(f)\big)$.

\vspace{0.05cm}
Write $C_1= B_1g_0$ and $C_2= B_2g_0$. Since $B_1 \cup B_2$ is a basis for $W$ and $g_0$ is an automorphism of $W$, it follows that $C_1\cup C_2$ is also a basis for $W$ (cf. \cite[Theorem 2.4(3)]{s-roman07}).

\vspace{0.05cm}
By $\rm(iii)$, since the subspace $T_f$ of $V$ is a transversal of $\ker(f)$, we see that the corestriction of $f_{|_{T_f}}$ to $R(f)$ is an isomorphism. Denote the inverse of this corestriction map by $g_1$. It is clear that $g_1\colon R(f)\to T_f$ is an isomorphism. Write $C'_1= B_1g_1$ and $C_3= B_3g_1$. Since $B_1 \cup B_3$ is a basis for $R(f)$ and $g_1\colon R(f)\to T_f$ is an isomorphism, it follows that $C'_1\cup C_3$ is a basis for $T_f$ (cf. \cite[Theorem 2.4(3)]{s-roman07}).

\vspace{0.05cm}
Write $W\cap T_f=T_{(f_{\upharpoonright_W})}$. By $\rm(ii)$, it is clear that $C'_1$ is a basis for the subspace $T_{(f_{\upharpoonright_W})}$ of $W$, and so $C'_1\subseteq \langle C_1\cup C_2\rangle$. Therefore $C_1\cup C_2\cup C_3$ is a basis for $W+T_f$. Extend $C_1\cup C_2\cup C_3$ to a basis $C:=C_1\cup C_2\cup C_3\cup C_4$ for $V$, where $C_4\subseteq V\setminus (W+T_f)$. Since $\codim(W+R(f))=\codim(W+T_f)$, we get $|B_4|=|C_4|$. Therefore there exists a bijection $g_2 \colon B_4\to C_4$. Define $g\in L(V)$ by
\begin{equation*}
	vg=
	\begin{cases}
		vg_0     & \text{if $v\in B_1\cup B_2$},\\
		vg_1     & \text{if $v\in B_3$},\\
		vg_2  & \text{if $v\in B_4$}.
	\end{cases}
\end{equation*}

It is routine to verify that $g$ is bijective, $g\in L_{\mathbb{S}(W)}(V)$, and $fgf = f$. Hence $f\in \ureg(L_{\mathbb{S}(W)}(V))$.
\end{proof}

%%%%%%%%%%%%%%%%%%%%%%%%%%%%%%%
%%%%%%%%%%%%%%%%%%%%%%%%%%%%%%%
%%%%%%%%%%%%%%%%%%%%%%%%%%%%%%%
%%%%%%%%%%%%%%%%%%%%%%%%%%%%%%%
%\newpage
If $I_W\in \mathbb{S}(W)$, then the following proposition gives a sufficient condition for $L_{\mathbb{S}(W)}(V)$ to be unit-regular.

\begin{proposition}\label{L(V,W)=unit-regular-2}
Let $\mathbb{S}(W)$ be a subsemigroup of $L(W)$ such that $I_W\in \mathbb{S}(W)$. If $\mathbb{S}(W)$ is a subgroup of $\aut(W)$ and $\codim(W)$ is finite, then $L_{\mathbb{S}(W)}(V)$ is unit-regular.
\end{proposition}

\begin{proof}[\textbf{Proof}]
Let $f\in L_{\mathbb{S}(W)}(V)$. Then $f_{\upharpoonright_W} \in \mathbb{S}(W)$. Since $\mathbb{S}(W)$ is a subgroup of $\aut(W)$, we have $f_{\upharpoonright_W}\in \aut(W)$. Therefore both $\rm(i)$ and $\rm(ii)$ of Theorem \ref{unit-regular elment} trivially hold. 

\vspace{0.05cm}
Now, by Lemma \ref{subspace-transversal}, let $T_f$ be a subspace of $V$ such that $T_f$ and $T_f\cap W$ are transversals of $\ker(f)$ and $\ker(f_{\upharpoonright_W})$, respectively. Observe that $f\colon T_f\to R(f)$ is an isomorphism and $Wf=W$. By \cite[Lemma 5.2]{shubh-lma22}, we therefore have $\dim(T_f/W)=\dim(R(f)/W)$. Since 
$f_{\upharpoonright_W}\in \aut(W)$, we simply have $W\subseteq T_f$ and $W\subseteq R(f)$. Then $W+T_f=T_f$, $W+R(f)=R(f)$, and both $T_f/W$ and $R(f)/W$ are subspaces of $V/W$. Therefore we obtain
\begin{align*}
	\codim(W+T_f)&=\codim(T_f)\\
	&=\dim(V/T_f)\\
	&= \dim((V/W)/(T_f/W))\\  
	&=\dim(V/W)-\dim(T_f/W) \qquad \qquad (\text{since } \codim(W)<\infty)\\
	&=\dim(V/W)-\dim(R(f)/W)\\% \; (\text{since } \dim(T_f/W)=\dim(R(f)/W))\\
	&= \dim((V/W)/(R(f)/W)) \;\;\qquad \qquad(\text{since } \codim(W)<\infty)\\
	&= \dim(V/R(f))\\
	&=\codim(R(f))\\
	&= \codim(W+R(f)),
\end{align*}
and so $f$ satisfies Theorem \ref{unit-regular elment}\rm(iii).
Thus $f\in \ureg(L_{\mathbb{S}(W)}(V))$ by Theorem \ref{unit-regular elment}. Hence, since $f$ is arbitrary, the semigroup $L_{\mathbb{S}(W)}(V)$ is unit-regular.
\end{proof}
%%%%%%%%%%%%%%%%%%%%%%%%%%%%%%%
%%%%%%%%%%%%%%%%%%%%%%%%%%%%%%%
%%%%%%%%%%%%%%%%%%%%%%%%%%%%%%%
%%%%%%%%%%%%%%%%%%%%%%%%%%%%%%%

%\newpage

If $I_W\in \mathbb{S}(W)$, then the following theorem determines when $L_{\mathbb{S}(W)}(V)$ is unit-regular.
%\begin{theorem}
%Let $\mathbb{S}(W)$ be a subsemigroup of $L(W)$ such that $I_W\in \mathbb{S}(W)$. Then $L_{\mathbb{S}(W)}(V)$ is unit-regular if and only if one of the following holds:
%\begin{enumerate}
%\item[\rm(i)] $\mathbb{S}(W)$ is a subgroup of $\aut(W)$ and $\codim(W) < \infty$.
%
%\item[\rm(ii)] $\mathbb{S}(W)$ is unit-regular, and $W=V$ or $W=\{0\}$ with $\dim(V) < \infty$.
%	\end{enumerate}
%\end{theorem}

\begin{theorem}
Let $\mathbb{S}(W)$ be a subsemigroup of $L(W)$ such that $I_W\in \mathbb{S}(W)$. Then $L_{\mathbb{S}(W)}(V)$ is unit-regular if and only if one of the following holds:
\begin{enumerate}
\item[\rm(i)] $\mathbb{S}(W)$ is a subgroup of $\aut(W)$ and $\codim(W)$ is finite.

\item[\rm(ii)] $\mathbb{S}(W)$ is unit-regular and $W=V$.
\end{enumerate}
\end{theorem}

\begin{proof}[\textbf{Proof}]
Suppose that $L_{\mathbb{S}(W)}(V)$ is unit-regular. By Proposition \ref{L(V,W)=unit-regular-2}, we see that $\rm(i)$ may hold. Let us assume that $\rm(i)$ does not hold.

\vspace{0.05cm}
Since $L_{\mathbb{S}(W)}(V)$ is unit-regular, it follows that $\mathbb{S}(W)$ is unit-regular by Proposition \ref{epimorphism-L-VW} and the fact that any homomorphic image of a unit-regular semigroup is unit-regular (cf. \cite[Proposition 2.7]{hick97}).

\vspace{0.05cm}

To show $W=V$, suppose to the contrary that $W\neq V$. Then we claim that $\mathbb{S}(W)$ is not a subgroup of $\aut(W)$. Suppose to the contrary that $\mathbb{S}(W)$ is a subgroup of $\aut(W)$. Then we further claim that $\codim(W)$ is finite. Suppose to the contrary that $\codim(W)$ is infinite. Then there exists a linear map $\psi\colon V/W\to V/W$ which is injective but not surjective. Take a basis $B_1$ for $W$. Extend $B_1$ to a basis $B_1\cup B_2$ for $V$, where $B_2\subseteq V\setminus \langle B_1\rangle$. Define $f\in L(V)$ by
\begin{align*}
vf=
\begin{cases}
	v    & \text{if $v\in B_1$,}\\
	v'   & \text{if $v\in B_2$, where $(v+W)\psi=v'+W$.}	
\end{cases}
\end{align*}
It is easy to see that $f$ is injective but not surjective. Therefore $\nul(f) \neq \corank(f)$, and so $f\notin \ureg(L(V))$ by \cite[Corollary 5.7]{shubh-lma22}.
Clearly $f\in L_{\mathbb{S}(W)}(V)$ and $L_{\mathbb{S}(W)}(V)\subseteq L(V)$, it follows that $f\notin \ureg(L_{\mathbb{S}(W)}(V))$, which  contradicts the unit-regularity of $L_{\mathbb{S}(W)}(V)$. Hence  $\codim(W)$ is finite, and so $\rm(i)$ holds. This is a contradiction to the assumption that $\rm(i)$ does not hold. Thus $\mathbb{S}(W)$ is not a subgroup of $\aut(W)$. Hence, since $L_{\mathbb{S}(W)}(V)$ is regular, we have $W = V$ by Theorem \ref{regular semigroup}.

\vspace{0.1cm}
For the converse, suppose first that $\rm(i)$ holds. Then $L_{\mathbb{S}(W)}(V)$ is unit-regular by Proposition \ref{L(V,W)=unit-regular-2}. Next, suppose that $\rm(ii)$ holds. If $W=V$, then $L_{\mathbb{S}(W)}(V)=\mathbb{S}(V)$, and so $L_{\mathbb{S}(W)}(V)$ is unit-regular by hypothesis.
\end{proof}

The following theorem gives a necessary and sufficient condition for $L_{\mathbb{S}(W)}(V)$ to be an inverse semigroup.

\begin{theorem}\label{inverse-semigroup}
Let $\mathbb{S}(W)$ be a subsemigroup of $L(W)$. Then $L_{\mathbb{S}(W)}(V)$ is an inverse semigroup if and only if
\begin{enumerate}
\item[\rm(i)] $\mathbb{S}(W)$ is an inverse semigroup;
\item[\rm(ii)] either $W=V$ or $\dim(V)=1$.
\end{enumerate}	
\end{theorem}

\begin{proof}[\textbf{Proof}]
Suppose that $L_{\mathbb{S}(W)}(V)$ is an inverse semigroup. Then $\rm(i)$ is true by Proposition \ref{epimorphism-L-VW} and the fact that any homomorphic image of an inverse semigroup is an inverse semigroup (cf. \cite[Theorem 5.1.4]{howie95}). 

\vspace{0.05cm}
The case $W = V$ may be possible, since $L_{\mathbb{S}(W)}(V) = \mathbb{S}(W)$ if $W = V$. Assume that $W\neq V$. Let $\alpha \in E(\mathbb{S}(W))$. Let $B_1$ and $B_2$ be bases for $N(\alpha)$ and $R(\alpha)$, respectively. Then $W = N(\alpha) \oplus R(\alpha)$, and so $B_1\cup B_2$ is a basis for $W$. Extend $B_1\cup B_2$ to a basis $B:=B_1\cup B_2\cup B_3$ for $V$, where $B_3\subseteq V\setminus W$. Claim that $|B_3|=1$. Suppose to the contrary that there are distinct $v_1, v_2\in B_3$. Define $f,g\in L(V)$ by
\begin{eqnarray*}
vf=	
\begin{cases}
	v\alpha   & \text{if $v\in B_1\cup B_2$},\\
	v_1       & \text{if $v\in B_3$}
\end{cases}
\qquad \text{ and } \qquad
vg=	
\begin{cases}
	v\alpha   & \text{if $v\in B_1\cup B_2$},\\
	v_2       & \text{if $v\in B_3$.}
\end{cases}	
\end{eqnarray*}
Clearly $f,g \in L_{\mathbb{S}(W)}(V)$. Since $\alpha$ is idempotent, it is routine to verify that $f$ and $g$ are idempotents. However, we have $v_1(fg)\neq v_1(gf)$, and so $fg\neq gf$. This leads to a contradiction, because $L_{\mathbb{S}(W)}(V)$ is an inverse semigroup. Hence $|B_3|=1$, and thus $\codim(W) = 1$. Write $B_3=\{u\}$.

\vspace{0.05cm}
Finally, we show that $W=\{0\}$. Suppose to the contrary that $W\neq \{0\}$. There are two cases to consider.

\vspace{0.05cm}
\noindent\textbf{Case 1:} $\alpha =I_W$.
Then $B_2\neq \varnothing$. Fix $w\in B_2$, and define $f,g\in L(V)$ by
\begin{eqnarray*}
vf=	
\begin{cases}
	v\alpha   & \text{if $v\in B_1\cup B_2$},\\
	0         & \text{if $v\in B_3$}
\end{cases}
\qquad \text{ and } \qquad	
vg=	
\begin{cases}
	v\alpha   & \text{if $v\in B_1\cup B_2$},\\
	w         & \text{if $v\in B_3$.}
\end{cases}	
\end{eqnarray*}
Clearly $f,g \in L_{\mathbb{S}(W)}(V)$. Since $\alpha$ is idempotent, it is routine to verify that $f$ and $g$ are idempotents. However, we have $u(fg)\neq u(gf)$, and so $fg\neq gf$. This leads to a contradiction, because $L_{\mathbb{S}(W)}(V)$ is an inverse semigroup.

\vspace{0.05cm}
\noindent\textbf{Case 2:} $\alpha \neq I_W$.
Then $B_1\neq \varnothing$. Fix $w'\in B_1$, and write $w'+u = u'$. Obviously $u'\neq u$ and $u'\in V\setminus W$. Define $f,g\in L(V)$ by
\begin{eqnarray*}
vf=	
\begin{cases}
	v\alpha   & \text{if $v\in B_1\cup B_2$},\\
	u         & \text{if $v\in B_3$}
\end{cases}
\hspace{5mm} \text{and} \hspace{5mm}	
vg=	
\begin{cases}
	v\alpha   & \text{if $v\in B_1\cup B_2$},\\
	u'        & \text{if $v\in B_3$}.
\end{cases}	
\end{eqnarray*}

Clearly $f,g \in L_{\mathbb{S}(W)}(V)$. Since $\alpha$ is idempotent, it is routine to verify that $f$ and $g$ are idempotents. However, we have $u(fg)=(uf)g=ug=u'$ and $u(gf)=(ug)f=u'f= (w'+u)f = w'f + uf = u$, and so $fg\neq gf$. This leads to a contradiction, since $L_{\mathbb{S}(W)}(V)$ is an inverse semigroup.

\vspace{0.05cm}
Thus, in either case, we get a contradiction. Hence $W=\{0\}$. Thus, since $\codim(W) = 1$, we conclude that $\dim(V)=1$.

\vspace{0.1cm}
Conversely, suppose that the given conditions hold. If $W=V$, then $L_{\mathbb{S}(W)}(V)=\mathbb{S}(W)$, and so $L_{\mathbb{S}(W)}(V)$ is an inverse semigroup by $\rm(i)$. Assume that $W\neq V$. Then $\dim(V)=1$, and so we get $W=\{0\}$. It follows that $L_{\mathbb{S}(W)}(V)=L(V)$, and so $L_{\mathbb{S}(W)}(V)$ is regular (cf. \cite[p.63, Exercise 19]{howie95}). Since $\dim(V)=1$, we also see that $L_{\mathbb{S}(W)}(V)$ has only two idempotents, namely the identity and the zero linear transformations on $V$. Since these idempotents commute, we conclude that $L_{\mathbb{S}(W)}(V)$ is an inverse semigroup.
\end{proof}

The next proposition is a consequence of Theorem \ref{inverse-semigroup}, which determines when $L(V)$ is an inverse semigroup.

\begin{proposition}
$L(V)$ is an inverse semigroup if and only if $\dim(V)\leq 1$.
\end{proposition}

\begin{proof}[\textbf{Proof}]
Suppose that $L(V)$ is an inverse semigroup. The result is obvious if $V=\{0\}$. Let us assume that $V\neq \{0\}$. Take $W=\{0\}$. Then $L_{L(W)}(V)= L(V)$, and hence $\dim(V)=1$ by Theorem \ref{inverse-semigroup}. 

\vspace{0.1cm}
Conversely, suppose that $\dim(V)\leq 1$. The result is obvious if $\dim(V)=0$. Let us assume that $\dim(V)=1$. Take $W=\{0\}$. Then $L_{L(W)}(V)= L(V)$ and $L(W)$ is an inverse semigroup. Hence, since $\dim(V)=1$, we conclude that $L(V)$ is an inverse semigroup by Theorem \ref{inverse-semigroup}.

\end{proof}

%%%%%%%%%%%%%%%%%%%%%%%%%%%%%%%
%%%%%%%%%%%%%%%%%%%%%%%%%%%%%%%
%%%%%%%%%%%%%%%%%%%%%%%%%%%%%%%
%%%%%%%%%%%%%%%%%%%%%%%%%%%%%%%
%\newpage
The following theorem gives a necessary and sufficient condition for $L_{\mathbb{S}(W)}(V)$ to be completely regular.

\begin{theorem}\label{completely-regular-semigroup}
Let $\mathbb{S}(W)$ be a subsemigroup of $L(W)$. Then $L_{\mathbb{S}(W)}(V)$ is a completely regular semigroup if and only if $\mathbb{S}(W)$ is a completely regular semigroup and one of the following holds:
\begin{enumerate}
\item[\rm(i)] $W=V$.
\item[\rm(ii)] $\codim(W)=1$ and $\mathbb{S}(W)$ is a subgroup of $\aut(W)$.
\end{enumerate}	
\end{theorem}

\begin{proof}[\textbf{Proof}]
Suppose that $L_{\mathbb{S}(W)}(V)$ is completely regular. Then $\mathbb{S}(W)$ is completely regular by Proposition \ref{epimorphism-L-VW} and the fact that any homomorphic image of a completely regular semigroup is completely regular (cf. \cite[Lemma II.2.4]{pet-reil99}).

\vspace{0.05cm}
The case $W = V$ may be possible, since $L_{\mathbb{S}(W)}(V) = \mathbb{S}(W)$ if $W = V$. Assume that $W\neq V$. First, we show that $\codim(W)=1$. Let $B_1$ be a basis for $W$. Extend $B_1$ to a basis $B_1\cup B_2$ for $V$, where $B_2\subseteq V\setminus W$. Claim that $|B_2|=1$. Suppose to the contrary that there are distinct $v_1, v_2\in B_2$. Take $\alpha \in \mathbb{S}(W)$, and then fix $w\in R(\alpha)$. Define $f\in L(V)$ by
\begin{equation*}
vf=	
\begin{cases}
	v\alpha    & \text{if $v\in B_1$},\\
	v_1        & \text{if $v=v_2$},\\
	w\alpha    & \text{if $v=B_2\setminus \{v_2\}$}.
\end{cases}
\end{equation*}
Clearly $f \in L_{\mathbb{S}(W)}(V)$. Moreover, we have $v_1, w\in R(f)$ and $v_1f=wf$. Therefore $R(f)$ is not a transversal of $\ker(f)$, and so $f$ does not belong to any subgroup of $T(V)$ (cf. \cite[Theorem 2.10]{clifford61}). Thus $f$ is not contained in any subgroup of $L_{\mathbb{S}(W)}(V)$, which  contradicts the completely regularity of $L_{\mathbb{S}(W)}(V)$. Hence $|B_2|=1$, and thus $\codim(W)=1$.

\vspace{0.05cm}
Finally, we show that $\mathbb{S}(W)$ is a subgroup of $\aut(W)$. It is obvious if $W=\{0\}$. Assume that $W\neq \{0\}$. Let $\beta\in \mathbb{S}(W)$. Claim that $\beta$ is injective. Suppose to the contrary that $w_1\beta=w_2\beta$ for some distinct $w_1,w_2\in W$. Recall that $\mathbb{S}(W)$ is completely regular. It follows that $\beta$ is contained in some subgroup of $\mathbb{S}(W)$, and so $\beta$ is contained in some subgroup of $T(W)$. Therefore $R(\beta)$ is a transversal of $\ker(\beta)$ (cf. \cite[Theorem 2.10]{clifford61}). Choose $w_0\in R(\beta)$ such that $w_0$ belongs to the $\ker(\beta)$-class that contains $w_1$ and $w_2$. We may assume that $w_0\neq w_1$. Define $g\in L(V)$ by 
\begin{equation*}
vg=	
\begin{cases}
	v\beta      & \text{if $v\in B_1$},\\
	w_1         & \text{if $v\in B_2$}.
\end{cases}
\end{equation*}
Clearly $g \in L_{\mathbb{S}(W)}(V)$. However, we have $w_0,w_1\in R(g)$ and $w_0g=w_1g$. Therefore $R(g)$ is not a transversal of $\ker(g)$, and so $g$ does not belong to any subgroup of $T(V)$ (cf. \cite[Theorem 2.10]{clifford61}). Thus $g$ is not contained in any subgroup of $L_{\mathbb{S}(W)}(V)$, which  contradicts the completely regularity of $L_{\mathbb{S}(W)}(V)$. Hence $\beta$ is injective. Further, we see that $\beta$ is surjective, since $\beta$ is injective and $R(\beta)$ is a transversal of $\ker(\beta)$. Thus $\mathbb{S}(W)$ is a subsemigroup of $\aut(W)$. Hence, since $\mathbb{S}(W)$ is completely regular, we conclude that $\mathbb{S}(W)$ is a subgroup of $\aut(W)$ by Lemma \ref{r-subsemigroup=subgroup}.

\vspace{0.1cm}
Conversely, suppose that the given conditions hold. If $W=V$, then $L_{\mathbb{S}(W)}(V)=\mathbb{S}(W)$, and so $L_{\mathbb{S}(W)}(V)$ is completely regular. Assume that $W\neq V$. Then, by (ii), $\codim(W)=1$ and $\mathbb{S}(W)$ is a subgroup of $\aut(W)$. Take a basis $B$ for $W$. Then $B\cup \{x\}$ is a basis for $V$, where $x\in V\setminus W$. For $z\in V$ and $\lambda\in \mathbb{S}(W)$, let $\alpha_z^{\lambda}\in L_{\mathbb{S}(W)}(V)$ be such that  $(\alpha_z^{\lambda})_{\upharpoonright_W} = \lambda$ and $x\alpha_z^{\lambda} = z$. Notice that $L_{\mathbb{S}(W)}(V)=\{\alpha_z^{\lambda}\colon z \in V \text{ and } \lambda \in \mathbb{S}(W)\}$, and for all $y\in W$, $z\in V$, and $\lambda, \delta\in \mathbb{S}(W)$,  
\[\alpha_x^{\lambda}\alpha_x^{\delta} = \alpha_x^{\lambda\delta} \;\;\text{ and }\;\;\alpha_y^{\lambda}\alpha_z^{\delta} = \alpha_{y\delta}^{\lambda\delta}.\]

\vspace{0.05cm}
Let $f\in L_{\mathbb{S}(W)}(V)$. Then $f = \alpha_y^{\lambda}$ for some $y\in V$ and $\lambda \in \mathbb{S}(W)$. There are two cases to consider.

\vspace{0.05cm}
\noindent\textbf{Case 1:} $y\in W$. Recall that $\mathbb{S}(W)$ is a subgroup of $\aut(W)$. Define $\varepsilon = \alpha_{y\lambda^{-1}}^{I_W}$ and $g = \alpha_{y\lambda^{-2}}^{\lambda^{-1}}$, where $\lambda^{-1}$ denotes the inverse of $\lambda$ in $\mathbb{S}(W)$. Obviously $\varepsilon, g\in L_{\mathbb{S}(W)}(V)$. Moreover, we obtain
\begin{align*}
\varepsilon\varepsilon & = \alpha_{y\lambda^{-1}}^{I_W} \alpha_{y\lambda^{-1}}^{I_W} = \alpha_{y\lambda^{-1}}^{I_W} = \varepsilon,\\
f\varepsilon & = \alpha_y^{\lambda} \alpha_{y\lambda^{-1}}^{I_W} = \alpha_y^{\lambda} = f,\\
\varepsilon f & =  \alpha_{y\lambda^{-1}}^{I_W}\alpha_y^{\lambda} = \alpha_y^{\lambda} = f,\\
g\varepsilon & = \alpha_{y\lambda^{-2}}^{\lambda^{-1}} \alpha_{y\lambda^{-1}}^{I_W} = \alpha_{y\lambda^{-2}}^{\lambda^{-1}} = g,\\
\varepsilon g & =  \alpha_{y\lambda^{-1}}^{I_W}\alpha_{y\lambda^{-2}}^{\lambda^{-1}} = \alpha_{y\lambda^{-2}}^{\lambda^{-1}} = g,\\
fg & = \alpha_y^{\lambda} \alpha_{y\lambda^{-2}}^{\lambda^{-1}} = \alpha_{y\lambda^{-1}}^{I_W} = \varepsilon,\\
gf & = \alpha_{y\lambda^{-2}}^{\lambda^{-1}}\alpha_y^{\lambda}  = \alpha_{y\lambda^{-1}}^{I_W} = \varepsilon.
\end{align*}
Hence the subsemigroup $\langle f, g\rangle$ of $L_{\mathbb{S}(W)}(V)$ generated by $\{f,g\}$ is a subgroup of $L_{\mathbb{S}(W)}(V)$.

\vspace{0.05cm}
\noindent\textbf{Case 2:} $y\in  V\setminus W$. Then, since $\mathbb{S}(W)$ is a subgroup of $\aut(W)$ and $\lambda\in \mathbb{S}(W)$, we see that $f$ is bijective. Therefore $f\in U(L_{\mathbb{S}(W)}(V))$. 

\vspace{0.1cm}
Thus, in either case, we see that $f$ belongs to a subgroup of $L_{\mathbb{S}(W)}(V)$. Hence, since $f$ is arbitrary, we conclude that $L_{\mathbb{S}(W)}(V)$ is completely regular.  
\end{proof}

%%%%%%%%%%%%%%%%%%%%%%%%%%%%%%%
%%%%%%%%%%%%%%%%%%%%%%%%%%%%%%%
%%%%%%%%%%%%%%%%%%%%%%%%%%%%%%%
%%%%%%%%%%%%%%%%%%%%%%%%%%%%%%%

The next proposition is a consequence of Theorem \ref{completely-regular-semigroup}, which determines when $L(V)$ is a completely regular semigroup.

\begin{proposition}
$L(V)$ is a completely regular semigroup if and only if $\dim(V)\leq 1$.	
\end{proposition}

\begin{proof}[\textbf{Proof}]
Suppose that $L(V)$ is completely regular. The result is obvious if $V=\{0\}$. Let us assume that $V\neq \{0\}$. Take $W=\{0\}$. Then $L_{L(W)}(V) = L(V)$, and so $\codim(W)=1$ by Theorem \ref{completely-regular-semigroup}. Hence $\dim(V)=1$.

\vspace{0.05cm}
Conversely, suppose that $\dim(V)\leq 1$. The result is obvious if $\dim(V)=0$. Let us assume that $\dim(V)=1$. Take $W = \{0\}$. Then obviously $L(W)$ is a group, $\codim(W)=1$, and $L_{L(W)}(V) = L(V)$. Hence $L(V)$ is completely regular by Theorem \ref{completely-regular-semigroup}.
\end{proof}

%%%%%%%%%%%%%%%%%%%%%%%%%%%%%%%
%%%%%%%%%%%%%%%%%%%%%%%%%%%%%%%
%%%%%%%%%%%%%%%%%%%%%%%%%%%%%%%
%%%%%%%%%%%%%%%%%%%%%%%%%%%%%%%


\begin{thebibliography}{10}
	
	\bibitem{chaiya-sf19}
	Y. ~Chaiya.
	\newblock{Natural partial order and finiteness conditions on semigroups of linear transformations with invariant subspaces}.
	\newblock{\em Semigroup Forum}, 99: 579--590, 2019.
	
	
	\bibitem{clifford61}
	A. ~H. ~Clifford and ~G. ~B. ~Preston.
	\newblock{\em The Algebraic Theory of Semigroups, Volume I}.
	\newblock{American Mathematical Society, Number 7 in Mathematical Surveys}, 1961.
	
	
	\bibitem{hick97}
	J. ~B. Hickey and ~M. ~V. ~Lawson.
	\newblock{Unit regular monoids}.
	\newblock{\em Proceedings of the Royal Society of Edinburgh: Section A Mathematics}, 127(1): 127--143, 1997.
	
	
	\bibitem{hony11}
	P. ~Honyam and ~J. ~Sanwong.
	\newblock{Semigroups of transformations with invariant set}.
	\newblock{\em Journal of the Korean Mathematical Society}, 48(2): 289--300, 2011.
	
	
	\bibitem{howie95}
	J. ~M. ~Howie.
	\newblock{\em Fundamentals of Semigroup Theory}.
	\newblock{\em Volume 12 of London Mathematical Society Monographs}.
	\newblock{Oxford University Press, New York}, 1995.
	
	
	\bibitem{koni-sf22}
	J. ~Konieczny. 
	\newblock{Semigroups of transformations whose restrictions belong to a
		given semigroup}.
	\newblock{\em Semigroup Forum}, 104(1): 109--124, 2022.
	
	
	\bibitem{magill66}
	K. ~D. ~Magill ~Jr.
	\newblock{Subsemigroups of {S(X)}}.
	\newblock{\em Mathematica Japonica}, 11: 109--115, 1966.
	
	
	\bibitem{nenth-ijmms06}
	S. ~Nenthein and ~Y. ~Kemprasit.
	\newblock{On transformation semigroups which are {$\mathcal{B}\mathcal{Q}$}-semigroup}.
	\newblock{\em International Journal of Mathematics and Mathematical Sciences}, 2006: 1--10, 2006.
	
	
	\bibitem{nenth05}
	S. ~Nenthein, ~P. ~Youngkhong, and ~Y. ~Kemprasit.
	\newblock{Regular elements of some transformation semigroups}
	\newblock{\em Pure Mathematics and Applications}, 16(3): 307--314, 2005.
	
	
	\bibitem{pet-reil99}
	M. ~Petrich and ~N. ~R. ~Reilly.
	\newblock{\em Completely Regular Semigroups}.
	\newblock{John Wiley and Sons, New York}, 1999.
	
	
	\bibitem{s-roman07}
	S. ~Roman.
	\newblock{\em Advanced Linear Algebra}.
	\newblock{\em volume 135 of Graduate Texts in Mathematics}.
	\newblock{Springer-Verlag, New York}, 3rd edition, 2008.
	
	
	\bibitem{shubh-lma22}
	M. ~Sarkar and ~S. ~N. ~Singh.
	\newblock{Unit-regular and semi-balanced elements in various semigroups of transformations}.
	\newblock{https://arXiv.org/abs/2106.08063v2}.
	
	
	\bibitem{shubh-sf22-2}
	M. ~Sarkar and ~S. ~N. ~Singh.
	\newblock{Comments on \lq\lq Semigroups of transformations whose restrictions belong to a given semigroup\rq\rq}.
	\newblock{\em Semigroup Forum}, 104(3): 758--759, 2022.
	
	\bibitem{shubh-aejm22}
	M. ~Sarkar and ~S. ~N. ~Singh.
	\newblock{On certain semigroups of transformations with an invariant set}.
	\newblock{\em Asian-European Journal of Mathematics}, 15(11): 2250198, 2022.
	
	
	\bibitem{sun-ams15}
	L. ~Sun and J. ~Sun.
	\newblock{A note on naturally ordered semigroups of transformations with invariant set}.
	\newblock{\em Bulletin of the Australian Mathematical Society}, 91(2): 264--267, 2015.
	
	
	\bibitem{sun-as13}
	L. ~Sun and L. ~Wang.
	\newblock{Natural partial order in semigroups of transformations with invariant set}.
	\newblock{\em Bulletin of the Australian Mathematical Society}, 87(1): 94--107, 2013.
	
	
\end{thebibliography}
\end{document}